\documentclass[11pt,twoside,english]{amsart}
  \usepackage{amsmath}
  \usepackage{amsfonts}
  \usepackage{amsthm}
  \usepackage{amssymb}
  \usepackage[all]{xy}
  \usepackage{verbatim}  
  \usepackage{longtable}
  \usepackage{stmaryrd}   
  \usepackage{mathrsfs}   
  \usepackage{color}
  \usepackage{enumitem}
  \usepackage{graphicx}
  \usepackage{url}
  \usepackage{mathtools}

  \usepackage[center,loose,nooneline]{subfigure}
  
\makeatletter \theoremstyle{plain}
\newtheorem{thm}{Theorem}

\newtheorem*{thm*}{Theorem}
\newtheorem{lem}{Lemma}[section]
\newtheorem{thm2}[lem]{Theorem}
\newtheorem{cor}[thm]{Corollary}
\newtheorem{prop}[lem]{Proposition}

\theoremstyle{remark}
\newtheorem{rmk}[lem]{Remark}

\theoremstyle{definition}

 \newcommand{\N}{\mathbb{N}} 
 
 \newcommand{\R}{\mathbb{R}}

 \newcommand*\colvec[3][]{\begin{pmatrix}\ifx\relax#1\relax\else#1\\\fi#2\\#3\end{pmatrix}}

 \newcommand{\Hhh}{\mathscr{H}}

 \newcommand{\diam}{\mathrm{diam}}
 \newcommand{\dist}{\mathrm{dist}}
 \newcommand{\dm}{\mathrm{d}}

  \newcommand{\0}{\mathbf{0}}

  \newcommand{\woz}{\backslash\{0\}}
  \newcommand{\wo}{\backslash}

\newcommand\ndot{{\mkern 2mu\cdot\mkern 2mu}} 
 
 \newcommand{\lm}{\varnothing}

  \newcommand{\g}{\mathscr{G}} 
  \newcommand{\eucl}{\text{\scalebox{0.8}{\,$\mathbb{E}$}}}  
  
 \newcommand{\Norm}{\|\ndot\|} 
\newcommand{\tnorm}{{{}^{\text{\scalebox{0.8}{$\|\ndot\|$}}}}} 
 \newcommand{\hnorm}{\text{\scalebox{0.7}{$\|\ndot\|$}}} 
 \newcommand{\henorm}{\text{\scalebox{0.7}{$|\ndot |$}}}
  \newcommand{\eNorm}{|\ndot |}

   \def\Ker{\operatorname{Ker}}

     \def\diam{\operatorname{diam}}
     
      \newcommand{\aprec}{\text{\scalebox{0.5}[1.4]{$\prec$}}\,}
  \newcommand{\asucc}{\,\text{\scalebox{0.5}[1.4]{$\succ$}}}
  \newcommand{\ascal}{\text{\scalebox{0.5}[1.4]{$\prec$}}\,\ndot,\ndot\,\text{\scalebox{0.5}[1.4]{$\succ$}}}

\begin{document}

\title{Marstrand-type projection theorems for linear projections and in normed spaces}
\author{Annina Iseli}

\address {Mathematisches Institut,
Universit\"at Bern,
Sidlerstrasse 5,
CH-3012 Bern,
Switzerland}

\email{annina.iseli@math.unibe.ch}


\thanks{ This research was supported by the Swiss National Science Foundation Grant Nr. 200020 165507 .}
\begin{abstract}
We establish Marstrand-type as well as Besicovich-Federer-type projection theorems for closest-point projections onto hyperplanes in the normed space $\R^{n}$.
In particular, we prove that if a norm on $\R^{n}$ is $C^{1,1}$-regular, then the analogues of the well-known statements from the Euclidean setting hold. On the other hand, we construct an example of a $C^{1}$-regular norm in $\R^2$ for which Marstrand-type theorems fail. These results are obtained by comparison arguments. \end{abstract}
\maketitle

\section{Introduction}

This paper is concerned with the behavior of Hausdorff measure and Hausdorff dimension under projections along linear foliations of $\R^n$ and in finite dimensional normed spaces. 
 Let $A\in \R^2$ be a Borel set. By $\dim A$ denote its Hausdorff dimension and by $\Hhh^s$ its Hausdorff $s$-measure where $s>0$. For every angle $\theta\in [0,\pi)$ consider the orthogonal projection $P^\eucl_\theta:\R^2\to L_\theta$ of $\R^2$ onto the line $L_\theta=\{r(\cos\theta, \sin \theta):r\in \R\}\subset \R^2$. From the facts that $L_\theta$ is a set of dimension $1$ and that the projection $P^\eucl_\theta$ is a {$1$-Lipschitz} mapping, one easily deduces that $\dim P^\eucl_\theta A\leq \min\{1,\dim A\}$ for all $\theta\in [0,\pi)$.
In 1954, Marstrand~\cite{Marstrand1954} proved that given a Borel set $A\in \R^2$, the orthogonal projection of $A$ onto the line $L_\theta$ is a set of Hausdorff dimension $$\dim P^\eucl_{L_\theta}A=\min\{1,\dim A\}$$ for $\Hhh^1$-a.e $\theta\in [0,\pi)$, i.e., given a Borel set $A$, there exists an $\Hhh^1$-zero set $E\subset [0,\pi)$ such that $\dim P^\eucl_\theta A= \min\{1,\dim \}$ for all angles $\theta\in [0,\pi)\wo E$. This theorem marked the start of a long sequence of results in the same spirit. They are known as Marstrand-type projection theorems and we summarize some of them in Theorem~\ref{thm_euclidean} below. 
In order to formally make sense of Theorem~\ref{thm_euclidean} recall the following definitions. For positive integers $m<n$ we denote by $G(n,m)$  the Grass-mannian manifold, i.e. the family of $m$-dimensional linear subspaces ($m$-planes) of $\R^n$. 
For every $m$-plane $V\in G(n,m)$, let $P^\eucl_V:\R^n\to V$ be the orthogonal projection of $\R^n$ onto $V$.  We will refer to the set $\{P^\eucl_V:V\in G(n,m)\}$ as the family of orthogonal projections (onto $m$-planes). Notice that the Grassmannian $G(n,m)$ is equipped with a natural measure $\sigma_{n,m}$ which is induced by the action of~$O(n)$ on $G(n,m)$ and the invariant Haar measure on~$O(n)$; 
see~\cite[Chapter\,3]{Mattila1995}. Moreover, since~$G(n,m)$ carries a (smooth) manifold structure the notion of Hausdorff dimension of subsets of $G(n,n-1)$ is well-defined.\\[0.3cm]
The following Theorem is a summary of results due to Marstrand~\cite{Marstrand1954}, Kaufman~\cite{Kaufman1968}, Mattila~\cite{Mattila1975}, Falconer~\cite{Falconer1982}, and Peres-Schlag~\cite{PS2000}.

\begin{thm2}\label{thm_euclidean}
For each $m$-plane $V\in G(n,m)$ denote by $P^\eucl_V:\R^n\to V$ the orthogonal projection of~$\R^n$ onto $V$. Then, for all Borel sets $A\subseteq \R^n$, the following hold:
 \begin{enumerate}[label={\arabic*.}, topsep=5pt, itemsep=5pt, leftmargin=30pt]
\item If $\dim A \leq m$, then 
\begin{enumerate}[label={\alph*)}, topsep=3pt, itemsep=3pt]
\item  $\dim P^\eucl_V A= \dim A$ for $\sigma_{n,m}$-a.e.\ $V\in G(n,m)$,
\item For $0<\alpha\leq\dim A$, $\dim\,\{V\in G(n,m): \dim(P^\eucl_V A)<\alpha\}\leq (n-m-1)m+\alpha$.
\end{enumerate}
\item If $\dim A > m$, then 
\begin{enumerate}[label={\alph*)}, topsep=3pt, itemsep=3pt]
\item  $\Hhh^m P^\eucl_V A>0$ for $\sigma_{n,m}$-a.e.\ $V\in G(n,m)$,
\item $ \dim\,\{V\in G(n,m): \mathscr{\Hhh}^m(P^\eucl_V A)=0\}\leq (n-m)m+m-\dim A$.
\end{enumerate} 
\item If $\dim A>2m$, then
\begin{enumerate}[label={\alph*)}, topsep=3pt, itemsep=3pt] 
\item  $P^\eucl_V A\subset \R^m$ has non-empty interior for $\sigma_{n,m}$-a.e.\ $V\in G(n,m)$,
\item $\dim\, \{V\in G(n,m):  \text{the interior of } P^\eucl_V A \text{ is empty }  \} \leq (n-m)m-\dim A +2m$
\end{enumerate}
\end{enumerate}
\end{thm2}

Many of the above statements are proven to be sharp; see e.g.~\cite{KaufMat1975, Falconer1982, Falconer1885_book}. Similar problems have been studied in various settings such as the Heisenberg groups~\cite{BFMT2012, BDCFMT2013, Hovila2014} and Riemannian surfaces~\cite{BaloghIseli2018_2, BaloghIseli2016, Brazilian2016}. Moreover, for an overview on the numerous works on the topic of projection theorems, we recommend the textbooks~\cite{Mattila1995, Falconer_Book, Mattila2015} as well as the survey articles~\cite{Mattila2004, Mattila_Arx2017}.\\[0.3cm] 
Another important projection theorem with a rather different flavor relates the size of sets under projections to their rectifiability properties. Recall that a subset $A$ of a metric space $(X,d)$ is called $m$-rectifiable if there exist a collection of at most countably many Lipschitz mappings $f_i:\R^m\to X$ such that 
$\Hhh^m\Big(A\, \wo \bigcup_i f_i(\R^m)\Big)=0.$ On the other hand, a set $E\subseteq \R^n$ is called purely $m$-unrectifiable, if $\Hhh^m(E\cap A)=0$ for every $m$-rectifiable set $A\subseteq\R^n$. The following theorem is due to Besicovitch~\cite{Besic1939} and Federer~\cite{Federer1947}; see also~\cite{Mattila1995}

\begin{thm2}\label{thm_besfed}
An $\Hhh^m$-measurable set $A\subseteq \R^n$ with $\Hhh^m(A)<\infty$ is purely $m$-unrectifiable if and only if $\Hhh^m(P_V^\eucl(A))=0$ for $\sigma_{n,m}$-a.e.\ $V\in G(n,m)$. \end{thm2}

Theorem~\ref{thm_besfed} has been generalized to other settings such as families of transversal projections in metric spaces~\cite{HJJL2012} and families of projections in the Heisenberg group~\cite{Hovila2014}.\\[0.3cm]
In this paper, we establish versions of the above Theorems for families of linear and surjective projections and families of closest-point projections with respect to sufficiently regular norms on $\R^n$. These results improve parts of the results in~\cite{BaloghIseli2018} jointly achieved with Balogh.\\[0.3cm]
We call a family of mappings $\{P_V: V\in G(n,m)\}$ a family of linear and surjective  projections (onto $m$-planes), if for every $V\in G(n,m)$, $P_V:\R^n\to V$ is a linear and surjective mapping. Notice that the family of orthogonal projections $\{P^\eucl_V:V\in G(n,m)\}$ is a family of linear and surjective projections. Moreover, every linear and surjective projection $P_V:\R^n\to V$ is a Lipschitz mapping. Hence it is a natural question whether Marstrand-type projection theorems generalize to families of linear and surjective projections.\\[0.3cm]
Many families of linear and surjective projections $\{P_V:V\in (n,m)\}$ are given in terms of linear foliations. Namely, if for $V\in G(n,m)$, we have $P_V(P_Vx)=P_Vx$ for all $x\in \R^n$, then there exists an $(n-m)$-plane $W\in G(n,n-m)$ with $V\cap W=\{0\}$  such that for all $x\in \R^n$, $P_Vx=a$, where $x=a+w$, $a\in V$ and $w\in W$. The affine $m$-planes $a+W$ with $a\in V$ are fibers of the foliation of $\R^n$ induced by $V$ and $W$, and it follows that  $\ker P_V=W$. 
It is straightforward to see that there exist families of linear and surjective projections for which Marstrand-type theorems must fail. Namely, consider $V_0\in G(n,m)$ and $W_0\in G(n,n-1)$ with $V_0\cap W_0=\{0\}$. Let $\mathcal{U}$ a small open neighbourhood of $V_0$ such that $V\cap W_0=\{0\}$ for all $V\in \mathcal{U}$. Now, for each $V\in \mathcal{U}$ define $P_V$ to be the projection onto~$V$ along the fibers $a+W_0$, i.e., $P_Vx=a$ where $x=a+w$, $a\in V$ and $w\in W_0$. Then, whenever the measure or dimension of a Borel set $A$ is decreased under $P_{V_0}$, then the same is true for all~$V\in\mathcal{U}$. Hence, Marstrand-type results must fail.\\
On the other hand, we shall prove that given a family $\{P_V:V\in G(n,m)\}$ of linear and surjective projections, if for every $V_0\in G(n,m)$ we can control the size of the set of $m$-planes $V\in G(n,m)$ for which $P_V$ and $P_{V_0}$ are projections along the same foliation, then Marstrand-type as well as Besicovitch-Federer-type theorems hold for this family. Define the mapping $\g:G(n,m) \to G(n,m)$ associated with the family $\{P_V:V\in G(n,m)\}$ by 
\begin{equation}\label{def_g(V)}
\g(V)=(\ker P_V)^\perp.
\end{equation}
This notation allows us to state the following analog of classical Marstrand-type projection theorems for families of linear and surjective projections.

\begin{thm}\label{thm_lin_proj}
Let $\{P_V:V\in G(n,m)\}$  be a family of linear and surjective projections whose associate mapping $\g$ is dimension non-decreasing and maps $\sigma_{n,m}$-positive sets to $\sigma_{n,m}$-positive sets. Then, the following hold for all Borel sets $A\subseteq \R^n$.
\begin{enumerate}[label={\arabic*.}, topsep=5pt, itemsep=5pt, leftmargin=30pt]
\item If $\dim A \leq m$, then 
\begin{enumerate}[label={\alph*)}, topsep=3pt, itemsep=3pt]
\item  $\dim P_V A= \dim A$ for $\sigma_{n,m}$-a.e.\ $V\in G(n,m) $,
\item For $0<\alpha\leq\dim A$, \\ $\dim\,\{V\in G(n,m) :  \dim(P_V A)<\alpha\}\leq (n-m-1)m+\alpha$.
\end{enumerate}
\item If $\dim A > m$, then 
\begin{enumerate}[label={\alph*)}, topsep=3pt, itemsep=3pt]
\item  $\Hhh^m (P_V A)>0$ for $\sigma_{n,m}$-a.e.\ $V\in G(n,m) $,
\item $\dim\,\{V\in G(n,m) : \mathscr{H}^m(P_V A)=0\}\leq (n-m)m+m-\dim A$.
\end{enumerate}
\item If $\dim A>2m$, then
\begin{enumerate}[label={\alph*)}, topsep=3pt, itemsep=3pt] 
\item  $P_V A\subseteq V\simeq \R^m$ has non-empty interior for $\sigma_{n,m}$-a.e.\ $V\in G(n,m)$,
\item  $\dim\,\{V\in G(n,m):  (P_VA)^\circ \neq \lm \}\leq (n-m)m+2m -\dim A.$
\end{enumerate}
\end{enumerate} 
\end{thm}

By the same methods we also obtain a Besicovitch-Federer-type projection theorem.

\begin{thm}\label{thm_besfed_lin}
Let $\{P_V:V\in G(n,m)\}$  be a family of linear and surjective projections such that for all $E\subset G(n,m)$, $\sigma_{n,m}(\g^{-1}(E))=0$ if and only if $\sigma_{n,m}(E)=0$. Then, an $\Hhh^m$-measurable set $A\subseteq \R^n$ with $\Hhh^m(A)<\infty$ is purely $m$-unrectifiable if and only if $\Hhh^m(P_V(A))=0$ for $\sigma_{n,m}$-a.e.\ $V\in G(n,m)$. 
\end{thm}

Although Theorems~\ref{thm_lin_proj} and~\ref{thm_besfed_lin} are interesting in their own right (see Section~\ref{sec_not_norm}) we are particularly interested in applying them to the setting of normed spaces. Let $\Norm$ be a strictly convex norm on $\R^n$, i.e., a norm whose unit sphere 
$S_\tnorm^{n-1}=\{x\in \R^n:\|x\|=1\}$ is the boundary of a strictly convex set. Then for every $x\in \R^n$ and every $m$-plane $V\in G(n,m)$, there exists a unique point $q\in V$ that realizes the distance between $x$ and $V$ with respect to $\Norm$, i.e., $\|x-q\|=\dist_\tnorm(x,V)\coloneqq\inf\{\|x-y\|:y\in V\}$.
We call the mapping $P^\hnorm_V:\R^n\to V$ given by $P_Vx=q$, where $\|x-q\|=\dist_\tnorm(x,V)$, the closest-point projection with respect to $\Norm$ onto $V$. Obviously, in case that $\Norm$ is the standard Euclidean norm $|\ndot |$ on $\R^n$ we have
$P^\eucl_V=P^\henorm_V$, for all $V\in G(n,m)$. We denote the unit sphere with respect to $|\ndot |$ in $\R^n$ by $S^{n-1}$. \\[0.3cm]
If $m=n-1$ and $\Norm$ is a strictly convex norm on $\R^n$, then one can check that $\{P_V^\hnorm:V\in G(n,m)\}$ is a family of linear and surjective projections (see Section~\ref{sec_norm}). If in addition, $\Norm$ is assumed to be $C^1$-regular (i.e.\ continuously differentiable outside of $\{0\}$), then at every point $x$ in the hypersurface  $S^{n-1}_\tnorm$, the unit outward normal $G(x)\in S^{n-1}$ in well-defined. This yields a mapping $G:S^{n-1}_\tnorm \rightarrow S^{n-1}$ that we call the Gauss map of $\Norm$. As we will show in Lemma~\ref{lem_Gauss}, the mapping $\g$ associated with the family $\{P_V\hnorm:V\in G(n,n-1) \}$ of linear and surjective projections can be expressed in terms of the inverse of $G$. This will allow us to prove the following results for families of projections onto hyperplanes.

\begin{thm}\label{thm_norm}
Let $\Norm$ be a strictly convex $C^1$-regular norm on $\R^n$.
If the Gauss map $G$ is dimension non-increasing and maps $\Hhh^{n-1}$-zero sets to $\Hhh^{n-1}$-zero sets, then the following hold for all Borel sets $A\subseteq \R^n$.
\begin{enumerate}[label={\arabic*.}, topsep=5pt, itemsep=5pt, leftmargin=30pt]
\item If $\dim A \leq n-1$, then 
\begin{enumerate}[label={\alph*)}, topsep=3pt, itemsep=3pt]
\item  $\dim (P_{w^\perp} A)= \dim A$ for $\Hhh^{n-1}$-a.e.\ $w\in S^{n-1}$,
\item For $0<\alpha\leq\dim A$, $\dim\,\{w\in S^{n-1} :  \dim (P_{w^\perp} A)<\alpha\}\leq \alpha$.
\end{enumerate}
\item If $\dim A > n-1$, then 
\begin{enumerate}[label={\alph*)}, topsep=3pt, itemsep=3pt]
\item  $\Hhh^{n-1} (P_{w^\perp} A)>0$ for $\Hhh^{n-1}$-a.e.\ $w\in S^{n-1}$,
\item $\dim\,\{w\in S^{n-1} : \mathscr{H}^{n-1}(P_{w^\perp} A)=0\}\leq 2(n-1)-\dim A$.
\end{enumerate}
\end{enumerate} 
\end{thm}

\begin{thm}\label{thm_besfed_norm}
Let $\Norm$ be a strictly convex $C^1$-regular norm on $\R^n$ such that for all $E\subset G(n,n-1)$, $\sigma_{n,n-1}(G(E))=0$ if and only if $\sigma_{n,n-1}(E)=0$. Then, an $\Hhh^{n-1}$-measurable set $A\subseteq \R^n$ with $\Hhh^{n-1}(A)<\infty$ is purely $m$-unrectifiable if and only if $\Hhh^{n-1}(P^\hnorm_V(A))=0$ for $\sigma_{n,n-1}$-a.e.\ $V\in G(n,n-1)$. 
\end{thm}

Note that if $\Norm$ is $C^{1,1}$-regular, then the Gauss map  $G$ (which essentially is the gradient of the norm) is locally Lipschitz and hence the requirements of Theorem~\ref{thm_norm} are satisfied. Thus the following corollary is a direct consequence of Theorem~\ref{thm_norm}.

\begin{cor}\label{cor_norm}
If $\Norm$ is a strictly convex $C^{1,1}$-regular norm on $\R^n$, then the conclusions of Theorem~\ref{thm_norm} hold for the projections $P_{w^\perp}^\hnorm:\R^n\to w^\perp$, $w\in S^{n-1}$.
\end{cor}

Exploiting the arguments from the proof of Theorem~\ref{thm_lin_proj} allows the construction of a $C^1$-regular norm on $\R^2$ for which Theorem~\ref{thm_norm} fails.

\begin{thm}\label{thm_counter}
There exists a $C^1$-regular norm $\Norm$ on $\R^2$ and a Borel set $A\subset \R^2$ with $\dim A\leq 1$ such that 
$\Hhh^1(\{w\in S^{n-1} :  \dim(P^\hnorm_{w^\perp} A)<\dim A\})>0.$ Thus, in particular, Conclusion~1 of Theorem~\ref{thm_norm} fails for $\Norm$.
Analogously, there exists a $C^1$-regular norm on $\R^2$ for which Conclusion~2 of Theorem ~\ref{thm_norm} fails.
\end{thm}

Theorem~\ref{thm_counter} shows that Marstrand-type projection theorems do not trivially hold for families of projections induced by a norm unless the norm is induced by a scalar product; see Section~\ref{sec_final}. Thereby, it underlines the importance of Theorem~\ref{thm_norm} and in some sense shows the sharpness of the regularity condition for the norm $\Norm$ in Theorem~\ref{thm_norm}. Comparing Corollary~\ref{cor_norm} with Theorem~\ref{thm_norm} raises the question whether or not Marstrand-type theorems hold for $C^{1,\delta}$-regular norms in $\R^2$ (i.e. derivatives of $\Norm$ of first order are locally $\delta$-H\"older).  Surprisingly, the answer to this question is related to the study of the structure of exceptional sets for Euclidean projections. We will address this relation in Section~\ref{sec_final}.\\[0.3cm]
The paper is organized as follows. In Section~\ref{sec_linear}, we prove Theorems~\ref{thm_lin_proj} and~\ref{thm_besfed_lin}. In Section~\ref{sec_norm}, we prove Theorems~\ref{thm_norm} and~\ref{thm_besfed_norm}  by applying Theorems~\ref{thm_lin_proj} and~\ref{thm_besfed_lin} respectively. Section~\ref{sec_not_norm} is for Propositions and Examples underlining the independent interest of Theoremd~\ref{thm_lin_proj} and~\ref{thm_besfed_lin}. In Section~\ref{sec_counter}, we explicitly construct a $C^1$-regular norm in $\R^2$ for which Theorem~\ref{thm_norm} fails and thereby prove 
Theorem~\ref{thm_counter}. Section~\ref{sec_final} is for final remarks.

\section{Linear projections}\label{sec_linear}

In this section we prove Theorem~\ref{thm_lin_proj}. Consider a family of linear and surjective projections $\{P_V: V\in G(n,m)\}$. Recall that we defined the mapping $\g:G(n,m)\to G(n,m)$ by $$\g(V)=(\Ker P_V)^\perp.$$
The following two lemmas will be used to compare the images of a Borel set under $P_V$ and $P^\eucl_{\g(V)}$.

\begin{lem}\label{lem_linear_1}
Let $f:\R^n\to\R^d$ and $g:\R^n\to \R^m$ be linear mappings with $\Ker f=\Ker g$. Then, there exists a bijective linear mapping $h:f(\R^n)\to g(\R^n)$ such that for all $x\in \R^n$, $h(f(x))=g(x)$. 
\end{lem}
  
\begin{proof}
In case $U\coloneqq \Ker f=\Ker g$ equals $\R^n$ or $\{0\}$, the Lemma is trivial. Therefore, we may assume without loss of generality that 
$0<k\coloneqq\dim(U)<n.$ Let $u_1,\ldots,u_k$ be a basis of $U$ and extend it to a basis $u_1,\ldots,u_k,w_1,\ldots,w_{n-k}$ of $\R^n$. 
Then, $f(w_1),\ldots,f(w_{n-k})$ is a basis of $f(\R^n)$ and $g(w_1),\ldots,g(w_{n-k})$ is a basis of 
$g(\R^n)$. 
Define a linear mapping $h:f(\R^n)\to g(\R^n)$ by setting $h(f(w_j))=g(w_j)$ for all $j=1,...,n-k$.
Then, $h$ is a bijection and for every $x\in\R^n$, $h(f(x))=g(x).$

\end{proof}

The following Lemma is a trivial consequence of Lemma~\ref{lem_linear_1}. It can be considered the key ingredient in the proofs of Theorem~\ref{thm_lin_proj} and Theorem~\ref{thm_besfed_lin}

\begin{lem}\label{lem_linear}
Let $f:\R^n\to\R^d$ and $g:\R^n\to \R^m$ be linear mappings  with $\Ker f=\Ker g$ and $A\subseteq \R^n$ a Borel sest. Then, $\dim f(A)=\dim g(A)$, and $\Hhh^{m}( f(A))=0$ if and only if $\Hhh^{m}( g(A))=0$.
\end{lem}

\begin{proof}[Proof of Theorem~\ref{thm_lin_proj}]
Let $A\subseteq \R^n$ be a Borel set  and $0<\alpha\leq \dim(A)\leq m$. We know that 1.a) and 1.b) of Theorem~\ref{thm_lin_proj} hold for $\{P_V^\eucl:V\in V\}$, that is,
\begin{equation}\label{know_1}
\sigma_{n,m} (\{  W\in G(n,m): \dim P^\eucl_W(A)< \alpha\} )=0
\end{equation}
\begin{equation}\label{know_2}
\dim\, \{W\in G(n,m): \dim P^\eucl_W(A)< \alpha\}\leq \alpha.
\end{equation}
By Lemma~\ref{lem_linear} with $f=P_V$ and $g=P_{\g(V)}^\eucl$, it follows that, for all $V\in G(n,m)$, 
\begin{equation}\label{eq_dim_equal}
\dim P_V(A)=\dim P^\eucl_{\g(V)}(A),
\end{equation}

On the other hand, notice that
\begin{equation}\label{eqn_proof_lin}\begin{split}
& \sigma_{n,m} (\g \{V\in G(n,m): \dim P^\eucl_{\g(V)}(A)< \alpha\})\\
&=\sigma_{n,m} (\{\g(V)\in G(n,m): \dim P^\eucl_{\g(V)}(A)< \alpha\})\\
& \leq \sigma_{n,m} (\{W\in G(n,m): \dim P^\eucl_W(A)< \alpha\})\\
\end{split}
\end{equation}

Thus, by \eqref{eq_dim_equal}, \eqref{eqn_proof_lin} and the fact that $\g$ does not map $\sigma_{n,m}$-positive sets to $\sigma_{n,m}$-zero sets it follows that
$ \sigma_{n,m} (\{V\in G(n,m): \dim P_V(A)< \alpha\})=0.$
This proves 1.a).\\[0.3cm]
Furthermore, combining \eqref{know_2} and \eqref{eq_dim_equal} with the fact that $\g$ is dimension non-decreasing, yields 
\begin{equation*}\begin{split}
\dim\,\{V\in G(n,m): \dim P_V(A)< \alpha\}&=\dim \,\{V\in G(n,m): \dim P^\eucl_{\g(V)}(A)< \alpha\}\\
&\leq \dim\,\g \{V\in G(n,m): \dim P^\eucl_{\g(V)}(A)< \alpha\}\\
&=\dim\,\{\g(V)\in G(n,m): \dim P^\eucl_{\g(V)}(A)< \alpha\}\\
& \leq \dim \,\{W\in G(n,m): \dim P^\eucl_W(A)< \alpha\}\, \leq \alpha.
\end{split}
\end{equation*}
This proves 1.b). The proofs of 2 and 3 are analogous.
\end{proof}

\begin{proof}[Proof of Theorem~\ref{thm_besfed_lin}]
Let $A\subset \R^n$ be $\Hhh^m$-measurable and $\Hhh^m(A)<\infty$. 
Let \begin{eqnarray*}
E=\{V\in G(n,m):\Hhh^m(P_V(A))>0\}\\
F=\{V\in G(n,m):\Hhh^m(P^\eucl_V(A))>0\}
\end{eqnarray*}
Recall that by Lemma~\ref{lem_linear} with $f=P_V$ and $g=P_V^\eucl$ we have $\Hhh^m(P_V(A))=0$ if and only if $\Hhh^m(P^\eucl_V(A))=0$. Thus, it follows that $E$ equals the preimage $g^{-1}(F)$. Hence Theorem~\ref{thm_besfed_lin} follows from Theorem~\ref{thm_besfed}.
\end{proof}
\begin{rmk}
The above proof reveals that the conditions for Theorem~\ref{thm_besfed_lin} can be slightly weakened. Namely, the following condition on $\g$ suffices for the conclusion of Theorem~\ref{thm_besfed_lin} to hold: For every $\Hhh^m$-measurable set $A\subset \R^n$ with $\Hhh^m(A)<\infty$, the set $F_A:=\{V\in G(n,m):\Hhh^m(P^\eucl_V(A))>0\}$ is a $\sigma_{n,m}$-zero set if and only if $g^{-1}(F_A)$ is a $\sigma_{n,m}$-zero set.
\end{rmk}

\section{Codimension-one projections in normed spaces}\label{sec_norm}

In this Section, we consider closest-point projections onto hyperplanes of $\R^n$ (i.e. $m=n-1$) that are induced by a norm.\\[0.3cm]
Recall that for a strictly convex norm $\Norm$ for every linear subspace $V\in G(n,m)$ the closest-point projection $P_V^\hnorm:\R^n\to V$ given by $\| P_vx-x\|=\dist_\tnorm(x,V)$. $x\in \R^n$, is well-defined. Notice, that for every point $x\in \R^n$ the point $P_Vx$ can be characterized as the unique point in the intersection $S^{n-1}_\tnorm(x,r)\cap V$, where $r=\dist_\tnorm(V,x)$.
 Now, in addition, assume that $\Norm$ is $C^1$-regular. Then, the unit sphere $S^{n-1}_\tnorm$ with respect to $\Norm$ is a compact $C^1$-hypersurface of $\R^n$ that admits an unit outward normal $G(x)\in S^{n-1}$ at every point $x\in S^{n-1}_\tnorm$. We call the map $G:S^{n-1}_\tnorm \to S^{n-1}$, $x\mapsto G(x)$ the Gauss map of $\Norm$. Notice that by the assumption of $C^1$-regularity of $\Norm$, $G$ is continuous.
Moreover, it has the following properties.

\begin{lem}\label{lem_Gauss}
Let $\Norm$ be a strictly convex $C^1$-regular norm on $\R^n$. Then the Gauss map $G:S^{n-1}_\tnorm\to S^{n-1}$ is a homeomorphism, $G(-v)=-G(v)$ and $\langle v, G(v)\rangle\neq 0$ for all $v\in S^{n-1}$.
\end{lem}

\begin{proof} 
Injectivity of $G$ follows immediately from strict convexity. To see this, assume that $G$ is not injective, thus, there exist two points $v,w\in S^{n-1}_\tnorm$, $v\neq w$, with  $G(v)=G(w)$. Hence for the tangent planes we have $T_vS^{n-1}_\tnorm=T_wS^{n-1}_\tnorm=:H$. Assume without loss of generality that $w$ lies on the same side of $H$ (if not, replace $w$ by $-w$).  
In case that $v+H=w+H$, strict convexity implies that $v=w$ which contradicts the choice of $v$ and $w$. Consider the case when $v+H\neq w+H$. Then, $H$, $v+H$ and $w+H$ are three parallel hyperplanes in $\R^n$. Moreover, by the assumption that $v$ and $w$ lie on the same side of $H$, $H$ is not the middle one of these three hyperplanes. 
Assume that $v+H$ is the middle one (the other case is analogous). Since $S^{n-1}_\tnorm\backslash \{v\}$ is a continuum containing $w$ and $-w$, the affine plane $v+H$ must intersect $S^{n-1}_\tnorm$ in more than one point. This contradicts strict convexity. Hence, it follows that $G$ is injective.\\[0.2cm]
Now, consider a direction $u\in S^{n-1}$ and let $V$ be its orthogonal complement. Since $S_\tnorm^{n-1}$ is compact, the set 
$\{t>0: S_{\tnorm}^{n-1}\cap (tu+V)\neq\lm\}$ has a maximum $t_0>0$. Thus, $H:=t_0u+V$ is the (affine) tangent plane of $S_{\tnorm}^{n-1}$ at the point $x$ where  $S_{\tnorm}^{n-1}$ intersects $L_u=\{tu:t\in \R\}$. Moreover, since $V$ was chosen to be orthogonal to $u$, it follows that $G(x)=u$. Hence, $G$ is surjective.\\[0.2cm]
Finally, notice that by antipodal symmetry of $\Norm$, that is $\|v\|=\|-v\|$ for all $v\in S^{n-1}$, it follows that $G(-v)=-G(v)$ for all $v\in S^{n-1}$.
\end{proof}

We will prove Theorem~\ref{thm_norm} by applying Theorem~\ref{thm_lin_proj}. Therefore, the following lemma is essential.

\begin{lem}\label{lem_norm_lin} 
For a strictly convex $C^1$-norm $\Norm$, $\{P_V^\hnorm:V\in G(n,n-1)\}$ is a family of linear and surjective projections. Moreover, for all $V\in G(n,n-1)$, $$\g(V)=\big( G^{-1}(w)\big)^\perp,$$ where $w=w(V)\in S^{n-1}$ orthogonal to $V$.
 \end{lem}

\begin{proof}  Let $V\in G(n,n-1)$. First, recall that for all $x\in \R^n\wo V$, $P^\hnorm_V(x)$ is the unique point in the intersection $S_\tnorm^{n-1}(x,r)\cap V$, where $r=\dist_{\tnorm}(x,V)$. Therefore, $V$ must be the tangent plane of $S_\tnorm^{n-1}(x,r)$ at $P^\hnorm_V(x)$; see Figure \ref{fig_gauss_proj}.

\begin{figure}[h]
\begin{center}
\def\svgwidth{250pt}
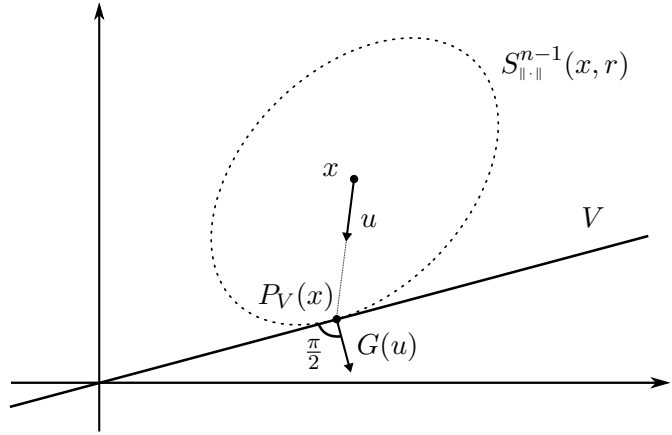
\end{center}
\caption{Gauss map and projections.}\label{fig_gauss_proj}
\end{figure}

However, this implies that the unit outward normal of  
$S_\tnorm^{n-1}(x,r)$ at $P^\hnorm_V(x)$ is orthogonal to $V$,
or, equivalently (see Figure~\ref{fig_gauss_proj}), that
$G(u)\perp V$, where $$u=\frac{P^\hnorm_V(x)-x}{\|P^\hnorm_V(x)-x\|}.$$
Let $w=w(V)\in S^{n-1}$ be a direction that is orthogonal to $V$, then for some $\lambda\in\{-1,1\}$, we have
$ G(u)=\lambda w.$
 Using the fact that $G$ is invertible and antipodally symmetric yields 
$u=\lambda G^{-1}(w). $
Thus, for every $x\in \R^n$, the projection direction $P_Vx-x$ is collinear with $u=G^{-1}(w)$ and $u$ does not depend on $x$ but only on $V$. Moreover, by Lemma~\ref{lem_Gauss}, $u=G^{-1}(w)$ is not contained in $V$. Hence, $P^\hnorm_V(x)$ is the unique intersection point of the affine line $x+L_u$ with the $m$-plane $V$ (recall that $L_v:=\{rv:r\in \R\}$ for all $v\in \R^n\woz$). This proves that $P^\hnorm_V:\R^n \to V$ is a linear and surjective mapping. Moreover, $(P^\hnorm _V)^{-1}(\{0\})=L_u$, and thus, 
$\g(V)=u^\perp=(G^{-1}(w))^\perp.$

\emph{ }

\vspace{-1cm}

\emph{ }
 \end{proof} 
 
\vspace{0.2cm}
 
Notice that in order to prove Theorem~\ref{thm_norm} it suffices to check that the map $\g$ associated with the family of closest-point projections with respect to $\Norm$ is dimension non-increasing and maps $\sigma_{n,n-1}$-positive sets to $\sigma_{n,n-1}$-positive sets. The main ingredient for this will be Lemma~\ref{lem_norm_lin}.  Lemma~\ref{lem_norm_lin} states that the associated mapping $\g:G(n,n-1)\to G(n,n-1)$ basically equals the inverse Gauss map $G^{-1}:S^{n-1}\to S^{n-1}_\tnorm$, once we identify hyperplanes $V\in G(n,n-1)$ by the outward normals $\{w,-w\}\subset S^{n-1}$. However, by our assumptions on $G$, the inverse Gauss map $G^{-1}$ has all the desired properties. In the below proof we carry out the details of this strategy.
 
 \begin{proof}[Proof of Theorem~\ref{thm_norm}]
Let $F\subseteq G(n,n-1)$ measurable. 
We will show that $\dim(\g(F))\geq \dim F$ and thereby establish that $\g$ is dimension non-decreasing.
Recall from the introduction that the notion of $\Hhh^s$-zero sets on $G(n,n-1)$ can be understood in terms of smooth chart maps for the Grassmannian manifold $G(n,n-1)$. From this fact, one easily deduces that a set $A\subset G(n,n-1)$ is an $\Hhh^s$-zero set in $G(n,n-1)$ if and only if $\{v\in S^{n-1}:v^\perp \in A\}$ is an $\Hhh^s$-zero set in $S^{n-1}$. Moreover, as a consequence of this equivalence, $\dim A =\dim \{v\in S^{n-1}:v^\perp \in A\}$. Thus, for our set $F$ it follows that
\begin{equation}\label{eqn1}
\dim \g(F)=\dim \{v\in S^{n-1}:v^\perp\in \g(F)\}.
\end{equation}

Recall that any norm on $\R^n$ is bi-Lipschitz equivalent to the Euclidean norm. In particular, so is our norm $\Norm$. This is equivalent to the fact that the mapping $S^{n-1}\to S_\tnorm^{n-1}$, $v\mapsto \frac{v}{\|v\|}$ is bi-Lipschitz equivalent (with respect to the Euclidean norm $|\ndot|$). Hence, for all sets
$A\subset S^{n-1}$, it follows that $\Hhh^s(A)=0$ if and only if $\Hhh^s(\{\frac{v}{\|v\|}\in S_\tnorm^{n-1}:v\in A\})=0$, for all $s>0$. Therefore, in particular, $\dim A=\dim \{\frac{v}{\|v\|}\in S_\tnorm^{n-1}:v\in A\}$. Combining this equality with \eqref{eqn1} yields
\begin{equation}\label{eqn2}
\dim \g(F)=\dim \{u\in S_\tnorm^{n-1}:u^\perp\in \g(F)\}.
\end{equation}

The condition that $u^\perp \in \g(F)$ in \eqref{eqn2} is equivalent to the existence of a hyperplane $V\in F$ for which $u^\perp =\g(V)$. However, by Lemma~\ref{lem_norm_lin}, the equality $u^\perp=\g(F)$ is equivalent to the equality $u=G^{-1}(w)$ where $w\in S^{n-1}$ with $w^\perp =V$. Plugging this into \eqref{eqn2} yields
\begin{equation}\label{eqn3}\begin{split}
\dim \g(F)&=\dim \{G^{-1}(w)\in S_\tnorm^{n-1}:w^\perp\in F)\}\\
&=\dim (G^{-1}\{w\in S^{n-1}:w^\perp\in F)\}).
\end{split}
\end{equation}

By Lemma~\ref{lem_Gauss}, $G$ is a homeomorphism and by our assumption it is dimension non-increasing. Thus, $G^{-1}$ is dimension non-decreasing homeomorphism. Hence, from \eqref{eqn3} and the argument above \eqref{eqn1}, it follows that
\begin{equation*}
\dim \g(F)\geq \dim \{w\in S^{n-1}:w^\perp\in F)\} = \dim F.
\end{equation*}
This proves that $\g$ is dimension non-decreasing.\\[0.3cm] 
Now we prove that $\g$ maps {$\sigma_{n,n-1}$-positive} sets to $\sigma_{n,n-1}$-positive sets.
Let $F\subset G(n,n-1)$ be measurable. It follows from the definition of $\sigma_{n,n-1}$ that 
\begin{eqnarray}
\sigma_{n,n-1} (F)&= &\Hhh^{n-1}(\{ v\in S^1: v^\perp \in F \})\\\
\sigma_{n,n-1} (\g(F))&=& \Hhh^{n-1}(\{ v\in S^1: v^\perp \in \g(F) \})
\end{eqnarray}
Then, by the arguments given above equations \eqref{eqn2} and \eqref{eqn3}, we may conclude
\begin{equation}\label{eqn6}\begin{split}
\sigma_{n,n-1} (\g(F))&=\Hhh^{n-1}( \{u\in S_\tnorm^{n-1}:u^\perp\in \g(F)\})\\
&=\Hhh^{n-1}(G^{-1}\{w\in S^{n-1}:w^\perp\in F)\}).
\end{split}
\end{equation}
Recall that $G$ is a homeomorphism that maps $\Hhh^{n-1}$-zero sets to $\Hhh^{n-1}$-zero sets. Hence, in case $\sigma_{n,n-1} (F)>0$ it follows that 
\begin{equation*}
\sigma_{n,n-1} (\g(F))=\Hhh^{n-1}(G^{-1}(\{w\in S^{n-1}:w^\perp\in F\}))>0.
\end{equation*}

\emph{}

\vspace{-20pt}

\end{proof}

\begin{proof}[Proof of Theorem~\ref{thm_norm}]
Given the above proof of Theorem~\ref{thm_norm}, in order to prove Theorem~\ref{thm_besfed_norm}, it suffices to check that $\g$ maps $\sigma_{n,n-1}$-zero sets to $\sigma_{n,n-1}$-zero sets.
Let $F\subset G(n,n-1)$ be a $\sigma_{n,n-1}$-zero set. Since zero sets are measurable, by \eqref{eqn6}, it follows that \begin{equation}\label{eqn7}
\sigma_{n,n-1} (\g(F))=\Hhh^{n-1}(G^{-1}\{w\in S^{n-1}:w^\perp\in F)\}).
\end{equation}
Recall that $G$ is a homeomorphism and that by assumption the preimages of zero sets under $G^{-1}$ are zero sets. Thus, by \eqref{eqn7}, it follows that  $\sigma_{n,n-1} (\g(F))=0$.
\end{proof}

\section{Linear projections that are not induced by a norm}\label{sec_not_norm}

In this section we emphasize the independent interest of Theorems~\ref{thm_lin_proj} and~\ref{thm_besfed_lin}. Namely, we will show that there exist many families of linear and surjective projections onto hyperplanes satisfying the conditions of Theorem~\ref{thm_lin_proj} that cannot be induced by a norm. First of all, notice that Theorems~\ref{thm_lin_proj} and~\ref{thm_besfed_lin} apply in all codimensions (i.e.\ for all $1\leq m<n$) while Theorems~\ref{thm_norm} and~\ref{thm_besfed_norm} only apply for codimension $1$ (i.e.\ $m=n-1$). Indeed, projections induced by a norm are in general not linear if the codimension is larger that $1$; see Section~\ref{sec_final}. In the sequel of this section, we will show that also for codimension $1$ there are many natural families of linear and surjective projections that are not induced by a norm.
\\[0.3cm]
Given a mapping $\g:G(n,m)\to G(n,m)$ we may define a family of linear and surjective projections
\begin{equation}\label{def_P_of_g}
P_V(x)=P_{\g(V)}^\eucl (x),
\end{equation}
$V\in G(n,m)$. Then, the associated mapping \eqref{def_g(V)} for this family of projections $\{P_V: V\in G(n,m)\}$ is the given mapping $\g$. Thus, if $\g$ is dimension non-decreasing and does not map $\sigma_{n,m}$-positive sets to $\sigma_{n,m}$-zero sets, then Theorem~\ref{thm_lin_proj} applies to the family $\{P_V: V\in G(n,m)\}$.
Notice that in order for a mapping $\g:G(n,n-1)\to G(n,n-1)$  to satisfy these conditions, properties such as continuity or injectivity are not required. However, for families of linear and surjective projections that are induced by a strictly convex $C^1$-norm it is known that $\g$ is given by the inverse Gauss map $G^{-1}$. Recall from Lemma~\ref{lem_Gauss} that $G^{-1}$ is known to be a homeomorphism in this setting. Therefore, we may conclude the following proposition.

\begin{prop}\label{prop1}
Every dimension non-decreasing mapping $\g:G(n,n-1)\to G(n,n-1)$ that does not map $\sigma_{n,n-1}$-positive sets to $\sigma_{n,n-1}$-zero sets and fails to be continuous and injective induces a family of linear and surjective projections that satisfies Theorem~\ref{thm_lin_proj} and is not given by a strictly convex norm on $\R^n$.
\end{prop}

Moreover, as the following Lemma shows, any mapping $\g:G(n,n-1)\to G(n,n-1)$ that is given in terms of the inverse Gauss map of a strictly convex $C^1$-norm possesses at least two fixed points.

\begin{lem}\label{lem_g_fixed}
There exist two vectors $v,w\in S^{n-1}_\tnorm$, $v\notin\{w,-w\}$, such that $G(v)=\frac{v}{|v|}$ and $G(w)=\frac{w}{|w|}$.
\end{lem}

\begin{proof}
Let $v_0\in S^{n-1}_\tnorm$ be a point that maximizes the Euclidean distance to the origin among all $v\in S^{n-1}_\tnorm$. 
Let ${\gamma:(-\epsilon,\epsilon)\to S^{n-1}_\tnorm}$ be a $C^1$-curve for which $\gamma(0)=v_0$. Thus, $\dot{\gamma}(0)\in T_{v_0} S^{n-1}_\tnorm$. Moreover, by choice of $v_0$ and the product rule for derivations, it follows that 
$$0=\frac{\dm}{\dm t}\langle \gamma(t),  \gamma(t) \rangle\,|_{t=0}=2\langle \dot{\gamma}(0),  \gamma(0) \rangle.$$ 
Since $G(v_0)$ is orthogonal to $T_{v_0}S^{n-1}_\tnorm$ it follows that $G(v_0)=\pm \frac{v_0}{|v_0|}$.
Since $G(v_0)$ points outward of $S^{n-1}_\tnorm$ at $v_0$, hence $G(v_0)=\frac{v_0}{|v_0|}$. Analogously, one proceeds for a point $w_0$ that minimizes the Euclidean distance to the origin among all $w\in S^{n-1}_\tnorm$. Then, unless $\Norm$ equals the Euclidean norm~$\eNorm$, we have $v_0\neq w_0$. Notice that for $\Norm=\eNorm$ the lemma is trivially true.
\end{proof}

Lemma~\ref{lem_g_fixed} immediately implies the following proposition.

\begin{prop}\label{prop2}
Every dimension non-decreasing mapping $\g:G(n,n-1)\to G(n,n-1)$ that does not map $\sigma_{n,n-1}$-positive sets to $\sigma_{n,n-1}$-zero sets and fails to have two fixed points, by \eqref{def_P_of_g} induces a family of linear and surjective projections that satisfies Theorem~\ref{thm_lin_proj} and is not given by a strictly convex norm on $\R^n$.
\end{prop}

Propositions~\ref{prop1} and~\ref{prop2} allow the construction of many families of linear and surjective projections that are not induced by a norm and for which Theorem~\ref{thm_lin_proj} holds.
In particular, it is easy to explicitly define and illustrate such families  in $\R^2$. Consider the following simple example. For every line $L\in G(2,1)$, let $\alpha(L)\in (0,\pi)$ be some angle and define $h(L)\in G(2,1)$ to be the line that makes a counter-clockwise angle $\alpha(L)$ with $L$. By definition of $h$, for every $L\in G(2,1)$ and every $x\in \R^2$, there exist unique unique $x_L\in L$ and $x_{h(L)}\in h(L)$ such that $x=x_L+x_{h(L)}$. Define $P_L:\R^2\to L$ by $P_Lx=x_L$; see Figure~\ref{fig_ex_linear}.
Notice that then $\g(L)=(h(L))^\perp$. Hence a line $L\in G(2,1)$ is a fixed point of $\g$ if and only if $\alpha(L)=\frac{\pi}{2}$. Therefore, if $\alpha(L)\neq \frac{\pi}{2}$ for all $L\in G(2,1)$, by Proposition~\ref{prop2}, the family $\{P_L:L\in (2,1)\}$ is not induced by a norm. In particular, if $\alpha$ is constant and not equal to $\frac{\pi}{2}$, then the family $\{P_L:L\in (2,1)\}$ is not induced by a norm and trivially satisfies Theorem~\ref{thm_lin_proj}.

\begin{figure}[h] 
\begin{center}
\def\svgwidth{170pt}
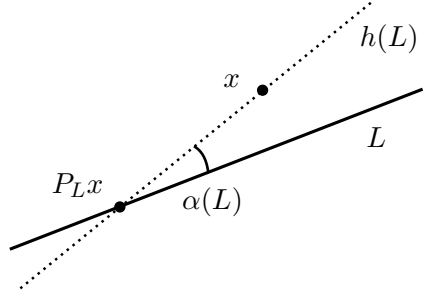
\end{center}
\caption{The projections $P_L:\R^2\to L$ induced by $h:G(2,1)\to (0,\pi)$.}\label{fig_ex_linear}
\end{figure}

\section{A norm for which Marstrand-type theorems fail}\label{sec_counter}

It is easy to construct families of linear and surjective projections for which Marstrand-type projection theorem fails. Similar examples are obtained from norms for which the Gauss map is not defined or multivalued for some points; see~\cite[Figures\,4\,and\,6]{BaloghIseli2018}. This raises the natural question, whether there exists a $C^1$- regular norm on $\R^n$ for which Marstrand-type theorems fail for projections onto hyperplanes. In this section, we will construct such a norm on $\R^2$ and thereby prove Theorem~\ref{thm_counter}.\\[0.3cm]
The following lemmas will be used in the proof of Theorem~\ref{thm_counter}. 

\begin{lem}\label{lem_counter_1}
For $0<d<2$, there exists a Borel set $A\subset \R^n$ of dimension $\dim A=d$ whose exceptional set $E=\{w\in S^1: \dim P^\eucl_{w^\perp} A< \min \{\dim A, 1\}$ for the family of orthogonal projections is a set of dimension $\dim E =d$
\end{lem}

\begin{proof}
Let $0<d< 1$. As established in~\cite{KaufMat1975} ($0<d<1$), there exists a compact set $A\subset \R^2$ of dimension $d$ such that the exceptional set $E=\{w\in S^1:  \dim(P^\eucl_{w^\perp}(A))< d\}$ is a set of dimension $\dim (E)=d$. Moreover, by~\cite{Kaufman1968} $E$ is a Borel set and by Marstrand's theorem it follows that $\Hhh^1(E)=0$. Let $1\leq d<2$, then by~\cite{Falconer1982}, there exists a compact set $A\subset \R^2$ of dimension $d$ such that the exceptional set $E=\{w\in S^1:  \Hhh^1 (P^\eucl_{w^\perp}(A))=0 \}$ is a set of dimension $\dim (E)=2-d>0$. Again, this set $E$ is a Borel set and by Marstrand's theorem it follows that $\Hhh^1(E)=0$.
\end{proof}

\begin{lem}\label{lem_counter_2}
Let $\Norm$ be a strictly convex $C^1$-regular norm on $\R^2$. Consider closest-point projections $P^\hnorm_{w^\perp}:\R^n \to w^\perp$, $w\in S^{n-1}$ and the Gauss map $G:S_\tnorm^1\to S^1$ . Let $0<d<1$ (resp. $1\leq d<2$) and let $A\subset \R^n$ and $E\subset S^1$ be the sets from  Lemma~\ref{lem_counter_1}. Let $E'=\{u\in S^1_\tnorm: \frac{u}{\|u\|}\in E\}$ Then, whenever $\Hhh^1 (G(E'))>0$, Conclusion~1 (resp. Conclusion~2) of Theorem~\ref{thm_norm} fails for $\Norm$.
\end{lem}

The  proof of Lemma~\ref{lem_counter_2} is very similar to the proofs of Theorem~\ref{thm_lin_proj} and Theorem~\ref{thm_norm}.

\begin{proof} 
Consider the case when $0<d<1$. By Lemma~\ref{lem_linear} (applied as in the proof of Theorem~\ref{thm_lin_proj}) and Lemma~\ref{lem_norm_lin} we have
\begin{equation}\begin{split}
\Hhh^1(\{v\in S^1:\dim P^\hnorm_{v^\perp}A< \dim A\}) 
&= \Hhh^1(\{w\in S^1:\dim P_{(G^{-1}(w))^\perp}A< \dim A\})\\
&= \Hhh^1(\{G(u)\in S^1:\dim P_{u^\perp}A< \dim A\})\\
&= \Hhh^1(G(\{u\in S_\tnorm^1:\dim P_{u^\perp}A< \dim A\}))\\
&= \Hhh^1(G(E'))>0.\\
\end{split}
\end{equation}
Hence, Conclusions~1.a) and~1.b) of Theorem~\ref{thm_norm} fail.
The case when $1\leq d<2$ is analogous. Then, Conclusions~2.a and~2.b of Theorem~\ref{thm_norm} fail.
\end{proof}

The following two lemmas outsource some technicalities from the proof of Theorem~\ref{thm_counter}.

\begin{lem}\label{lem_cov_arg}
Consider an interval $I\subset \R$ and two continuous curves $\alpha:I\to \R^m$ and $\beta:I\to \R^n$. Suppose that there exists a constant $M>0$ for which 
\begin{equation}\label{eq18}
 |\beta(s)-\beta (s')|\leq M|\alpha(s)-\alpha(s')|,
\end{equation}
 for all $s,s'\in I$.
 Then, for all Borel sets $F\subseteq [0,1]$ and for all $t>0$, \begin{equation}\label{eq19}
 \Hhh^t(\beta(F)) \leq (2M)^t\Hhh^t(\alpha(F)).
\end{equation} 
In particular, if follows that if $\Hhh^1(\beta(F))>0$, then $\Hhh^1(\alpha(F))>0$.
\end{lem}

We prove Lemma \ref{lem_cov_arg} by applying a straightforward covering argument based on the definition of the Hausdorff measure..

\begin{proof}
Let $t>0$ and $F\subseteq I$ a Borel set. In the case when $\Hhh^t(\alpha(F))=\infty$, \eqref{eq19} holds trivially. Therefore, we assume that $\Hhh^t(\alpha(F))=c$ where $0\leq c<\infty$. Let $\delta>0$. Then, there exists an open covering $\mathcal{A}:=\{A_i\}_{i=1}^N$ of $\alpha(F)$ where $N\in \N\cup\{\infty\}$ for which
 $\diam A_i\leq \delta $, for all $i=1,\ldots,N$ and $\sum_{i=1}^N\diam A_i^t\leq c+\delta $. Without loss of generality, assume that $A_i\cap \alpha(F)\neq \lm$ for all $i=1,\ldots,N$. Let $s_i\in I$ such that $\alpha(s_i)\in A_i\cap \alpha(F)$. Then, by \eqref{eq18}, the family of closed balls $B_i$ with center $\beta(s_i)$ and radius $M\diam A_i$ covers $\beta(F)$ and 
 $\diam B_i =2M\diam A_i \leq 2M\delta$ for all $i=1,\ldots,N$. This yields
 $$\Hhh^t_{2M\delta} (\beta(F)) \leq \sum_{i=1}^N(\diam B_i)^t\leq (2M)^t \sum_{i=1}^N(\diam A_i)^t\leq (2M)^t(c+\delta),$$ and hence 
 $\Hhh^t((\beta(I))\leq (2M)^t\,c$.
\end{proof}

The following lemma is an application of Lemma \ref{lem_cov_arg}.
\begin{lem}\label{lem_apply_covlem}
Let  $b\in (0,\infty]$ and let ${f,g:[0,b]\to [0,\infty)}$ be two strictly increasing functions. Define $h(t):=f(t)g(t)$ for all $t\in[0,b]$. Then, for all Borel sets $F\subseteq [0,b]$, if $\Hhh^1(f(F))>0$, then $\Hhh^1(h(F))>0$.
\end{lem}

\begin{proof}
Let $F\subseteq [0,b]$ be a Borel set with $\Hhh^1(f(F))>0$. Then, by sub-additivity of $\Hhh^1$ and the fact that $f$ is increasing, there exists a number $n\in \N$ with $n>\tfrac{1}{b}$, such that for $F_n:=F\cap [\tfrac{1}{n},b]$, we have $\Hhh^1(f(F_n))>0$.
For $s<s'\in [\tfrac{1}{n},b]$, we have
$$h(s')-h(s)=f(s')g(s')-f(s)g(s)\geq (f(s')-f(s))g(s')\geq g(\tfrac{1}{n}) f(s)-f(s')>0.$$
Applying Lemma~\ref{lem_cov_arg} for $\alpha=f:[\tfrac{1}{n},b]\to [0,\infty)$, $\beta=h:[\tfrac{1}{n},b]\to [0,\infty)$, and $M=\frac{1}{g(\frac{1}{n})}$, yields $\Hhh^1(h(F))\geq\Hhh^1(h(F_n))>0$.
\end{proof}

Our strategy for the proof of Theorem~\ref{thm_counter} goes as follows. For $0<d<1$ consider the Borel set $A\subset \R^2$ from Lemma~\ref{lem_counter_1} and its exceptional set $E=\{v\in S^1: \dim (P_{v^\perp} A) < \dim A\}$. We construct the norm $\Norm$ such that the Gauss map for $\Norm$ blows up the exceptional set $E$ to a set of positive $\Hhh^1$-measure. Thus, by Lemma~\ref{lem_counter_2}, Conclusion~1 of Theorem~\ref{thm_norm} fails for $\Norm$. The construction of such a norm $\Norm$ roughly goes as follows.
Identify~$S^1$ with the interval $[0,2\pi)$. This identification will be denoted by $\alpha^{-1}:S^1\to [0,2\pi)$. We consider a suitable subset $K\subset \alpha^{-1}(E)$ and construct a strictly increasing and continuous function $f$ that blows up the set $K$ to a set of positive length. Then, the integral $F$ of $f$ will be strictly convex and $C^1$. Now, we roll the graph of $F$ back up with $\alpha$ (resp.\ its extension~$h$); see Figure~\ref{fig_counterex_1}. Thus, the image $\Gamma$ of the graph of $F$ will be a piece of the boundary of a strictly convex set which defines a norm $\Norm$ on $\R^2$, see Figure~\ref{fig_counterex_4}. We will show that the Gauss map of this norm restricted to $\Gamma$, will still behave like the function $f$ in terms of its measure theoretic properties. (The case where $1\leq d<2$ is analogous.)

\begin{proof}[Proof of Theorem~\ref{thm_counter}] 
Let $0<d<1$ and consider the Borel set $A\subset \R^2$ from Lemma~\ref{lem_counter_1} and its exceptional set $E=\{v\in S^1: \dim (P_{v^\perp} A) < \dim A\}$.
Consider the parameterization $\alpha:[0,2\pi)\to S^1$ given by $\alpha(t):=(\cos(t),\sin(t))$. Since $\alpha$ is locally bi-Lipschitz, it follows that $\dim (\alpha^{-1}(E))=d$. Let $0<s<d$. Then, by definition of the Hausdorff dimension, $\Hhh^s(\alpha^{-1}(E))=\infty$. Therefore, by {\cite[Theorem\,8.13]{Mattila1995}}, there exists a compact set 
\begin{equation}\label{eq20}
K\subset \alpha^{-1}(E)= (\{ t\in [0,1]: \dim P^\eucl_{\alpha(t)^\perp}(A)<\dim A\}).
\end{equation}
 with $0<\Hhh^s(K)<\infty$. We assume without loss of generality that $K\subset [0,1]$. \\[0.3cm]
Now, define $f:[0,1]\to [0,1]$ by 
\begin{equation}\label{eq_f_in_counterex}
f(t):=\frac{1}{2}\Big(\frac{1}{\Hhh^s(K)}\Hhh^s\big([0,t]\cap K\big)+t\Big).
\end{equation}
Notice that $t\mapsto \Hhh^s([0,t])$ is non-decreasing and continuous. (In case $K$ is the triadic Cantor set the function then $t\mapsto \Hhh^s([0,t]\cap K$ is the triadic Cantor staircase function). Thus, $f$ is a strictly increasing homeomorphism. Furthermore, since $K$ is compact, $[0,1]\wo K$ consists of countably many (relatively) open intervals in~$[0,1]$. On each interval in~$[0,1]\wo K$, $f$ is linear with slope $\frac{1}{2}$. Hence, $\Hhh^1(f([0,1]\wo K))=\frac{1}{2}$ and it follows that $\Hhh^1(f(K))=\tfrac{1}{2}>0$. \\[0.3cm]
Define the mapping $F:[0,1]\to [0,1]$ by
 $$F(u):=\frac{1}{4}\int_0^u f(t) \dm t.$$
 
 Then, $F:[0,1]\to [0,1]$ is an injective and strictly convex $C^1$-mapping with $F(1)\leq \frac{1}{4}$. 
  Define $S\subset \R^2$ by $S:=\{r\left(\begin{smallmatrix}\cos(t)\\ \sin(t)\end{smallmatrix}\right):t\in [0,1], \, r\geq 0\}$. Moreover, we define the mapping $h:[0,1]\times[0,1]\to S$ by
  $h(x,y):=(1-y)\left(\begin{smallmatrix}\cos(x)\\ \sin(x)\end{smallmatrix}\right)$, and the curve $\gamma:[0,1]\to S$ by $\gamma(t):=h(t,F(t))$. Thus, the curve~$\gamma$ parameterizes the arc~$h(\mathrm{Graph}(F))$; see Figure~\ref{fig_counterex_1}.\medskip
  
\begin{figure}[h] 
\begin{center}
\def\svgwidth{300pt}
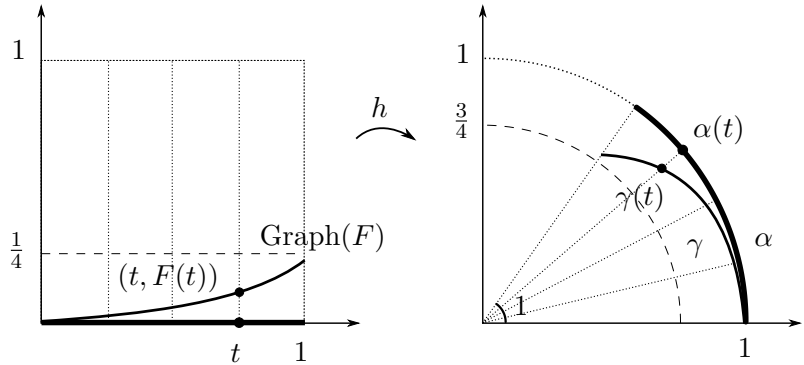
\end{center}
\caption{Construction of the curve $\gamma$ from $F:[0,1]\to [0,\tfrac{1}{4}]$. }\label{fig_counterex_1}
\end{figure}

 Observe that for all $t\in [0,1]$,
\begin{equation*}
  \alpha(t)=\frac{\gamma(t)}{|\gamma(t)|}.
\end{equation*}  
 Moreover, $\gamma$ is a regular $C^1$-curve and $\dot\gamma$ is given by
 \begin{eqnarray}
 \dot{\gamma}(t) & =& \left( \begin{matrix}
   -(1-F(t))\sin (t) & -\cos(t) \\
   (1-F(t))\cos(t)& -\sin(t) \\
  \end{matrix}  \right)  \left( \begin{matrix}
  1\\
   \frac{1}{4}f(t) \\
  \end{matrix} \right)\\  
  \label{eq21}
   & =& (1-F(t))\left( \begin{matrix}
   \cos (t+\frac{\pi}{2}) & -\sin (t+\frac{\pi}{2})  \\
   \sin (t+\frac{\pi}{2})& \cos (t+\frac{\pi}{2})  \\
  \end{matrix} \right)  \left(\begin{matrix}
  1\\
   \frac{1}{4(1-F(t))}f(t) \\
  \end{matrix}\right)_\cdot 
 \end{eqnarray}
 Notice that since $0\leq F(t)\leq \tfrac{1}{4}$, it follows that $4(1-F(t))\geq 3$ and 
\begin{equation}\label{eqn_new}
 \tfrac{1}{4(1-F(t))}\leq \tfrac{1}{3}.
\end{equation} 

Consider the curve $\beta:[0,1]\to S^1$, defined by $\beta(t):=\frac{\dot{\gamma}(t)}{|\dot{\gamma}(t)|}$. As we will establish later, $\beta$ has the following properties:
 \begin{enumerate}[label=(P\arabic*), topsep=3pt, itemsep=5pt, leftmargin=30pt]
 \item  \ $\beta:[0,1]\to S^1$ is an injective curve that travels in $S^1$ in counterclockwise direction from~$\beta(0)=\left(\begin{smallmatrix}0\\ 1 \end{smallmatrix}\right)$ to $\beta(1)$ where $\beta(1)=(\cos(s),\sin(s))$, with $s\in(\tfrac{\pi}{2},\pi)$.
 \item \ $\Hhh^1(\beta(K))>0$.
 \end{enumerate}

Denote the image of $[0,1]$ under $\gamma$ by $\Gamma$. From our bounds for the values of $\beta$ at $t=0$ and $t=1$ from property (P1), it follows that we can extend the union $\Gamma\cup (-\Gamma)$ to the image of a closed $C^1$-curve~$\bar{\Gamma}$, by gluing arcs $R$ and $-R$ to $\Gamma$ and $-\Gamma$, such that the tangential directions at the gluing points agree, as illustrated in Figure~\ref{fig_counterex_4}.

\begin{figure}[h] 
\begin{center}
\def\svgwidth{280pt}
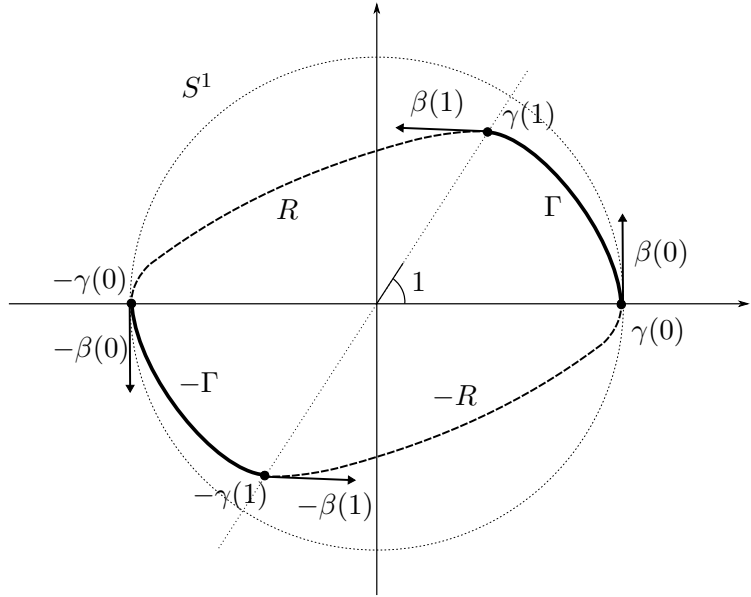
\end{center}
\caption{Normal sphere that contains the arc $\Gamma$.}\label{fig_counterex_4}
\end{figure}

Recall that by property (P1), $\beta$ is injective. Thus, $\bar{\Gamma}$ is a simply closed curve that bounds a strictly convex, antipodally symmetric subset of $\R^2$ with non-empty interior. Hence, $\bar{\Gamma}$ defines a norm $\Norm$ on $\R^2$ by setting $S_\tnorm^{1}:=\bar{\Gamma}$.
 Moreover, since $\beta(t)$ is tangential to $\bar{\Gamma}$ at $\gamma(t)\in \bar{\Gamma}$ for $t\in [0,1]$,  the Gauss map $G:S^1_\tnorm\to S^1$ of the norm $\Norm$ in such points is given by
\begin{equation}\label{eq22}
G(\gamma(t))=R_{\tfrac{\pi}{2}}\beta(t),
\end{equation}
where $R_{\tfrac{\pi}{2}}$ denotes the counterclockwise rotation about the angle $\tfrac{\pi}{2}$. Recall that by property~(P2), $\Hhh^1(\beta (K))>0$. 
Thus, \eqref{eq22} implies that $\Hhh^1 G(\gamma(t))>0$, where $G$ denotes the Gauss map of the arc $\Gamma$ parameterized by $\gamma$. 
This proves Theorem~\ref{thm_counter} given properties (P1) and (P2) for $\beta$.\\[0.3cm]
Thus, we are left to prove that $\beta$ actually does satisfy properties (P1) and (P2).
Let us begin by defining shorter notations for the objects appearing in \eqref{eq21}.
 For $t\in [0,1]$, we write 
$$M(t):=  \left( \begin{matrix}
   \cos (t+\frac{\pi}{2}) & -\sin (t+\frac{\pi}{2})  \\
   \sin (t+\frac{\pi}{2})& \cos (t+\frac{\pi}{2})  \\
  \end{matrix} \right) $$ and 
  $$v(t):= \left(\begin{matrix}
  1\\
   \frac{1}{4(1-F(t))}f(t) \\
  \end{matrix}\right)_\cdot$$ 
  Hence, $M(t)\in O(2)$, $v(t)\in (\{1\}\times [0,\tfrac{1}{3}])\subset \R^2$ (see \eqref{eqn_new}) and $\dot{\gamma}(t)=(1-F(t))M(t)v(t)$, for all $t\in [0,1]$. Set $w(t):=\frac{v(t)}{|v(t)|}$ for $t\in [0,1]$. Then, by \eqref{eq21}, and the fact that $M(t)\in O(2)$ for all $t\in[0,1]$, it follows that $\beta(t)=M(t)w(t)$.\smallskip
  
Recall that the functions $f:[0,1]\to [0,1]$ as well as $F:[0,1]\to [0,\tfrac{1}{4}]$ are strictly increasing. Thus, in particular, $t\mapsto  \tfrac{1}{4(1-F(t))}$ is strictly increasing. Also, recall that $\Hhh^1(f(K))>0$. Hence, the mapping
$\psi:[0,1]\to [0,\tfrac{1}{3}]$, defined by 
$$\psi(t):= \frac{1}{4(1-F(t))}f(t)$$ 
is strictly increasing as well. Moreover, Lemma~\ref{lem_apply_covlem} implies that $\Hhh^1( \psi(K))>0$. \\[0.2cm]
Note that $\R\to (\{1\}\times \R)\subset \R^2,x\mapsto \left(\begin{smallmatrix}1\\x\end{smallmatrix}\right) $ is an isometric embedding (i.e. a $1$-bi-Lipschitz mapping) and $v(t)= \left(\begin{smallmatrix}1\\\psi(t)\end{smallmatrix}\right)$. Therefore, 
$v:[0,1]\to \{1\}\times[0,\tfrac{1}{3}]$ is injective with 
$v(0)=\left(\begin{smallmatrix}1\\0\end{smallmatrix}\right)$ and $v(1)=\left(\begin{smallmatrix}1\\1/(4(1-F(1)))\end{smallmatrix}\right)$, and $\Hhh^1(v(K))>0$.\\[0.2cm]
Recall that $w(t)=\tfrac{v(t)}{|v(t)|}$, for $t\in [0,1]$. Thus, $w:[0,1]\to S^1$ is an injective curve that travels from $w(0)=v(0)=\left(\begin{smallmatrix}1\\0\end{smallmatrix}\right)$ to $w(1)$,  see Figure~\ref{fig_counterex_2}.\\

\begin{figure}[h] 
\begin{center}
\def\svgwidth{350pt}
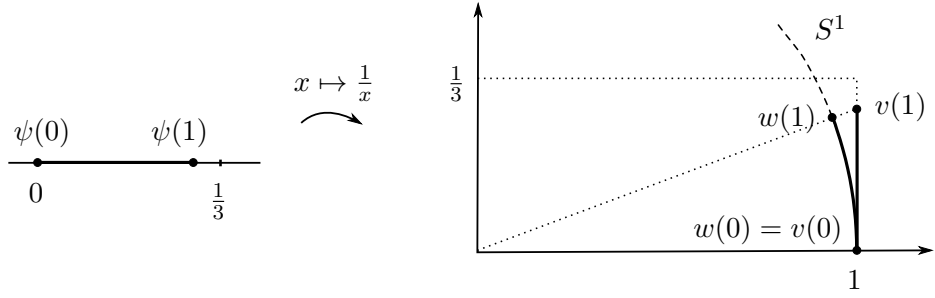
\end{center}
\caption{Construction of $v$ and $w$ from $\psi$.}\label{fig_counterex_2}
\end{figure}  

For $t\in [0,1]$, denote by $\theta(t)\in [0,2\pi)$ the counterclockwise angle from the $x$-axis to~$w(t)$, thus 
\begin{equation}\label{eq20c}
w(t)= \left(\begin{matrix}\cos(\theta(t))\\ \sin(\theta(t))\end{matrix}\right).
\end{equation}

Recall that $v(1)=\left(\begin{smallmatrix}1\\1/4(1-F(1)))\end{smallmatrix}\right)$ and notice that 
$$\frac{1}{1/(4(1-F(1)))}=4(1-F(1))\geq 3 > \frac{\cos(\frac{\pi}{2}-1)}{\sin(\frac{\pi}{2}-1)}_\cdot $$ 
Therefore, it follows that $w(1)=\frac{v(1)}{|v(1)|}=\left(\begin{smallmatrix}\cos(\theta(1))\\\sin(\theta(1))\end{smallmatrix}\right)$ with $\theta(1)\in(0,\tfrac{\pi}{2}-1)$.
Moreover, from the fact that $(\{1\}\times [0,\tfrac{1}{3}])\to S^1$, $x\mapsto \tfrac{x}{|x|}$ is a bi-Lipschitz mapping, it follows that $\Hhh^1(w(K))>0$.\\

Now, consider the curve $\beta:[0,1]\to S^1, \ t\mapsto M(t)w(t)$. The matrix $M(t)$ is the matrix of the counterclockwise rotation  about the angle $t+\tfrac{\pi}{2}$. 

Thus, it follows that 
\begin{equation}\label{eq20d}
\beta(t)=\left(\begin{matrix}\cos(t+\tfrac{\pi}{2}+\theta(t))\\ \sin(t+\tfrac{\pi}{2}+\theta(t))\end{matrix}\right).
\end{equation}

 This makes $\beta:[0,1]\to S^1$ an injective curve that travels in $S^1$ in counterclockwise direction from
 $\beta(0)= \left(\begin{smallmatrix}0\\1\end{smallmatrix}\right)$ 
 to $\beta(1)= \left(\begin{smallmatrix}\cos(s)\\ \sin(s)\end{smallmatrix}\right),$ where $s:=1+\tfrac{\pi}{2}+\theta(1)$ and thus $s\in ( 1+\tfrac{\pi}{2}, \pi)$; see Figure~\ref{fig_counterex_3}. This proves property (P1). 
 
\begin{figure}[h] 
\begin{center}
\def\svgwidth{300pt}
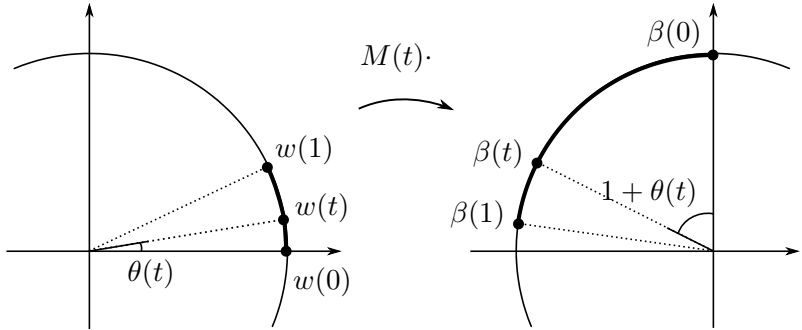
\end{center}
\caption{From $w(t)$ to $\beta(t)$ by left multiplication with $M(t)$, $t\in [0,1]$.}\label{fig_counterex_3}
\end{figure}  
 
Moreover, it follows from \eqref{eq20c} and \eqref{eq20d} that $|\beta(t)-\beta(t')|\geq |w(t)-w(t')|$, for all $t,t'\in [0,1]$. Thus, by Lemma \ref{lem_cov_arg} and the fact that $\Hhh^1(w(K))>0$, it follows that $\Hhh^1(\beta(K))>0$. This proves property~(P2).
 \end{proof}

\section{Final Remarks}\label{sec_final}

\subsection{Codimension greater than 1}

As pointed out in the introduction, for every strictly convex norm $\Norm$ on $\R^n$ and for every $1\leq m<n$, the family $P^\hnorm_V: \R^n\rightarrow \R^n$, $V\in G(n,m)$ of closest-point projections with respect to $\Norm$ is well-defined. 
Nevertheless, Theorem~\ref{thm_norm} only covers the case when $m=n-1$. We strongly believe that a statement similar to Theorem~\ref{thm_euclidean} holds for general codimension, i.e., for all $1\leq m<n$. However our methods do not allow a proof yet. One can check that, in general, for strictly convex norms (even if they have a good differentiable regularity) projections onto plane of codimension greater than one ($m<n-1$) are not linear mappings and therefore Theorem~\ref{thm_lin_proj} is not applicable. For example, a simple calculation (see~\cite[Section\,5.5]{AnninaPhD}) shows that projections onto lines induced by the $L_p$-norm on $\R^n$ for $n\geq 3$ are linear mappings if and only if $p=2$ (recall that the $L_2$ norm on $\R^n$ is the standard Euclidean norm).\\[0.2cm]
On the other hand, as we shall prove now, in case that a norm $\Norm$ on $\R^n$ is induced by an inner product space then all Euclidean projection theorems stated in the introduction hold for the family $\{P_V:V\in G(n,m)\}$ for all $1\leq m<n$. For this, denote the Euclidean inner product (the scalar product) in $\R^n$ by $\langle \ndot, \ndot \rangle$. Let $e_1,\ldots,e_n$ be the standard basis of $\R^n$ which is an orthonormal basis with respect to $\langle \ndot, \ndot \rangle$. Moreover, let  $\ascal$ be an inner product on $\R^n$ and $b_1,\ldots,b_n$ an orthonormal basis of $\R^n$ with respect to $\ascal$. Then, the linear mapping 
$\Psi:  (\R^n, \ascal) \to (\R^n, \langle \ndot, \ndot \rangle)$ defined by $\Psi(b_i)=e_i$ for all $i=1,\ldots,n$, is an isometry in the sense that $\aprec x,y\asucc =\langle \Psi(x), \Psi(y)\rangle$ for all $x,y\in \R^n$.  Let $x\in \R^n$ and $V\in G(n,m)$, then by definition of $P^\hnorm$, we have
$\| x-P^\hnorm_V(x)\|=\dist_\tnorm(V,x) $. Since $\Psi$ is an isometry, this implies that 
$|\Psi(x)-\Psi(P^\hnorm_V(x))|= \dist_\eucl(\Psi(x),\Psi(V))$, and hence, by definition of the Euclidean projection, $P^\eucl_{\Psi(V)}(\Psi(x))=\Psi(P^\hnorm_V(x))$. Hence, it follows that 
\begin{equation*}
P^\hnorm_V(x)=\psi^{-1}\circ P^\eucl_{\Psi(V)}\circ \Psi(x),
\end{equation*}
for all $x\in \R^n$ and $V\in G(n,m)$. 
Therefore, in particular, the projection $P_V^\hnorm:\R^n \to V$ is linear and surjective for all ${V\in G(n,m)}$. Moreover, the mapping $\g$ associated with the family $P^\hnorm_V:\R^n\to \R^n$, $V\in G(n,m)$ is given by $\Psi$. Since, $\Psi$ is a linear bijection, $\g:G(n,m)\to G(n,m)$ is a smooth diffeomorphism of manifolds and thus preserves zero-sets and Hausdorff dimension. Therefore, Theorem~\ref{thm_lin_proj} and Theorem~\ref{thm_besfed_lin} apply.

\subsection{Possibility for a generalization of Theorem~\ref{thm_counter}}

The main reason why we cannot state Theorem \ref{thm_counter} in any greater generality is lack of knowledge about the structure of exceptional sets for orthogonal projection.
Notice that the Gauss map $G:S^{1}_\tnorm\to S^{1}$ of the norm $\Norm$ constructed in the proof of Theorem \ref{thm_counter} might turn out to be a $\delta$-H\"older mapping for some $\delta>0$ depending on the geometry of $K$. This would then imply that there exists a $C^{1,\delta}$-regular norm for which Conclusions~1 and~2 of Theorem~\ref{thm_lin_proj} fail. For example, if $K$ happened to be the triadic cantor set, the mapping $f:[0,1]\to [0,1]$ defined in \eqref{eq_f_in_counterex} and therefore also the Gauss map $G:S^1_\tnorm\to S^1$ would be {$\tfrac{\log(2)}{\log(3)}$-H\"older} mappings.
 The question about the geometry of the exceptional sets is in general open. In particular, we do not know, whether a set like the triadic Cantor set appears as a subset of such exceptional sets. For a more detailed account on the study of the structure of exceptional sets for orthogonal projections we refer to the works \cite{JJLL2008, FaessOrp2014, OrpVen_Arx2017, Chen_Arx2017} and references therein.\\[0.2cm]
Furthermore, we do not know whether Theorem~\ref{thm_counter} generalizes to families of projections $P_V:\R^n\to V$, onto $(n{-}1)$-planes $V\in G(n,n-1)$. The main obstacle is that we do not have a suitable analog of the function $f$ given  in equation \eqref{eq_f_in_counterex} if $n\neq 2$. Notice that it is of great importance for the construction of $f$ that the continuity of $t \mapsto \Hhh^s([0,t]\cap K)$ is independent of the structure of $K$. However, the tentative higher-dimensional analog of this is not true. For example, if $K\subset [0,1]^2$ contains an isolated line segment parallel to an axis then the mapping $(t,r)\mapsto \Hhh^1\left( ([0,t]\times[0,r])\cap K\right)$ is not continuous.\\[0.2cm]
It is an interesting question whether Theorem~\ref{thm_besfed_norm} holds for the $C^1$-regular norms constructed in the proof of Theorem~\ref{thm_counter}. In order to approach this question, we suggest to study the rectifiability properties of the set sets $A$ of Lemma~\ref{lem_counter_1}. As pointed out in the proof of Lemma~\ref{lem_counter_1}, the construction of these sets are due to~\cite{KaufMat1975}. They are based on the number theoretic considerations in~\cite{Eggleston}.

\subsection{Projection theorems via differentiable transversality}

Peres and Schlag~\cite{PS2000} establish a very general projection theorem for families of (abstract) projections from compact metric spaces to Euclidean space. Their result states that if a sufficiently regular family of projections satisfies a certain transversality condition, then this yields bounds for the Sobolev dimension of the push-forward (by the projections) of certain measures. All the classical Marstrand-type projection theorems for orthogonal projection in $\R^n$ can be deduced as corollaries from their result; see~\cite[Section~6]{PS2000} and~\cite[Section~18.3]{Mattila2015}. Moreover, Hovila et.\ al.~\cite{HJJL2012} has proven that if a family of abstract projections satisfies transversality with sufficiently good transversality constants ,then this yields a Besicovitch-Federer-type projection theorem for this family of projections. This makes differentiable transversality a very powerful method in establishing  projection theorems in various settings. In particular, the works~\cite{Hovila2014} (Heisenberg groups) and~\cite{BaloghIseli2016} (Riemannian surfaces of constant curvature) are based on Peres and Schlag's notion of transversality.\\[0.2cm]
In fact, one can check that if a strictly convex norm $\Norm$ on $\R^n$ is $C^{2,\delta}$-regular for some $\delta>0$ then the induced family of closest-point projections satisfies differentiable transversality. Also, the better the regularity of the norm, the better the transversality constant. This is worked out in detail 
in~\cite{AnninaPhD}. Notice that the transversality constants affect the bounds for the size of the exceptional sets for Marstrand-type theorems deduced from transversality. Therefore, whenever $\Norm$ fails to be $C^\infty$-regular the Marstrand-type theorems that can be obtained by establishing differentiable transversality are worse than Theorem~\ref{thm_norm}. On the other hand, the fact that families of projections induced by a sufficiently regular norm are transversal to some extend can be considered a result of interest independent of projection theorems. Note that for example, families of closest-point projections in infinity dimensional Banach spaces fail to be transversal~\cite{Bate2017}.

\newpage
\bibliographystyle{abbrv}
\bibliography{literature_projections}

\end{document}